\documentclass[12pt]{amsart}

\usepackage{amssymb, amsfonts}
\usepackage{url}
\usepackage[all,arc]{xy}
\usepackage{amscd}
\usepackage{mathrsfs}
\usepackage{geometry}
\usepackage[colorlinks,citecolor=blue]{hyperref}
\usepackage{comment}
\usepackage{enumerate}
\usepackage{graphicx}
\usepackage{bm}
\usepackage{color}
\usepackage{anyfontsize}
\usepackage{caption}

\numberwithin{equation}{section}

\theoremstyle{plain}
\newtheorem{theorem}{Theorem}[section]
\newtheorem{corollary}[theorem]{Corollary}
\newtheorem{lemma}[theorem]{Lemma}
\newtheorem{proposition}[theorem]{Proposition}

\theoremstyle{remark}
\newtheorem{remark}{Remark}[section]
\newtheorem{example}[remark]{Example}

\begin{document}


\title[On static manifolds with boundary]{On static manifolds with boundary admitting a nowhere-vanishing static potential}

\author{Vladimir Medvedev}

\address{Faculty of Mathematics, National Research University Higher School of Economics, 6 Usacheva Street, Moscow, 119048, Russian Federation}

\email{vomedvedev@hse.ru}



\begin{abstract}
We study complete static manifolds with boundary admitting a nowhere-vanishing static potential. Our main result shows that, under a natural lower bound relating the scalar curvature and the boundary mean curvature, a simple static manifold with boundary must in fact have positive scalar curvature, negative boundary mean curvature, and be compact; we also obtain explicit relations and estimates involving the volume of the manifold and the geometry of its boundary. In the scalar-flat case, we prove global splitting and Ricci-flat rigidity results, including for disconnected boundary, while in the negative scalar curvature case we establish a sharp mean-curvature bound and characterize the equality case by an exponential warped-product structure. The proofs rely essentially on the study of the associated Einstein manifold. In appendix we derive several identities for static manifolds with boundary and discuss the associated Einstein manifold technique in the boundaryless setting.
\end{abstract}

\maketitle


\newcommand\cont{\operatorname{cont}}
\newcommand\diff{\operatorname{diff}}

\newcommand{\dvol}{\text{dA}}
\newcommand{\Ric}{\operatorname{Ric}}
\newcommand{\Hess}{\operatorname{Hess}}
\newcommand{\GL}{\operatorname{GL}}
\newcommand{\myO}{\operatorname{O}}
\newcommand{\myP}{\operatorname{P}}
\newcommand{\eye}{\operatorname{Id}}
\newcommand{\myF}{\operatorname{F}}
\newcommand{\Vol}{\operatorname{Vol}}
\newcommand{\odd}{\operatorname{odd}}
\newcommand{\even}{\operatorname{even}}
\newcommand{\ol}{\overline}
\newcommand{\mye}{\operatorname{E}}
\newcommand{\myo}{\operatorname{o}}
\newcommand{\myt}{\operatorname{t}}
\newcommand{\irr}{\operatorname{Irr}}
\newcommand{\mydiv}{\operatorname{div}}
\newcommand{\curl}{\operatorname{curl}}
\newcommand{\re}{\operatorname{Re}}
\newcommand{\im}{\operatorname{Im}}
\newcommand{\can}{\operatorname{can}}
\newcommand{\scal}{\operatorname{scal}}
\newcommand{\tr}{\operatorname{trace}}
\newcommand{\sgn}{\operatorname{sgn}}
\newcommand{\SL}{\operatorname{SL}}
\newcommand{\myspan}{\operatorname{span}}
\newcommand{\mydet}{\operatorname{det}}
\newcommand{\SO}{\operatorname{SO}}
\newcommand{\SU}{\operatorname{SU}}
\newcommand{\specl}{\operatorname{spec_{\mathcal{L}}}}
\newcommand{\fix}{\operatorname{Fix}}
\newcommand{\id}{\operatorname{id}}
\newcommand{\grad}{\operatorname{grad}}
\newcommand{\singsup}{\operatorname{singsupp}}
\newcommand{\wave}{\operatorname{wave}}
\newcommand{\ind}{\operatorname{ind}}
\newcommand{\mynull}{\operatorname{null}}
\newcommand{\inj}{\operatorname{inj}}
\newcommand{\arcsinh}{\operatorname{arcsinh}}
\newcommand{\Spec}{\operatorname{Spec}}
\newcommand{\Ind}{\operatorname{Ind}}
\newcommand{\Nul}{\operatorname{Nul}}
\newcommand{\inrad}{\operatorname{inrad}}
\newcommand{\mult}{\operatorname{mult}}
\newcommand{\Length}{\operatorname{Length}}
\newcommand{\Area}{\operatorname{Area}}
\newcommand{\Ker}{\operatorname{Ker}}
\newcommand{\floor}[1]{\left \lfloor #1  \right \rfloor}

\newcommand\restr[2]{{
  \left.\kern-\nulldelimiterspace 
  #1 
  \vphantom{\big|} 
  \right|_{#2} 
  }}


\section{Introduction and main results}

A \textit{static manifold} is a Riemannian manifold $(M^n,g)$, $n\geqslant 3$, admitting a nontrivial smooth function $V$, called a \textit{static potential}, which satisfies the \textit{static equation}
\begin{align}\label{static}
\operatorname{Hess}_gV-(\Delta_gV)g-V\operatorname{Ric}_g=0.
\end{align}
Taking the trace of this equation gives
\begin{align}\label{traced}
\Delta_gV=-\frac{R_g}{n-1}V.
\end{align}

A standard consequence of the static equation is that the scalar curvature $R_g$ is constant on every connected static manifold (see, e.g.,~\cite{fischer1975deformations,tod2000spatial,corvino2000scalar}). Thus \eqref{traced} implies that $V$ is a Laplace eigenfunction with eigenvalue $-\dfrac{R_g}{n-1}$.

A \textit{static manifold with boundary} is a triple $(M^n,g,V)$ satisfying the static equation~\eqref{static} in the interior together with the boundary condition
\begin{align}\label{staticb}
\partial_\nu V\,g_{\partial M}-VB_g=0\qquad\text{on }\partial M,
\end{align}
where $\nu$ is the outward unit normal, $g_{\partial M}$ is the induced metric, and $B_g$ is the second fundamental form of the boundary. \textit{In this paper we also assume that $(M,g)$ is complete up to the boundary. }

If $H_g:=\operatorname{tr}_{g_{\partial M}}B_g$ denotes the mean curvature, then tracing the boundary equation~\eqref{staticb} yields the \textit{Robin condition}
\begin{align}\label{tracedb}
\partial_\nu V=\frac{H_g}{n-1}V\qquad\text{on }\partial M.
\end{align}
As it is proved in~\cite[Proposition 1]{cruz2023static}, at every boundary point at which $V\neq 0$, the boundary is totally umbilical, i.e., $B_g=\dfrac{H_g}{n-1}g_{\partial M}$ Moreover, $H_g$ is constant on each connected component of the boundary. Thus if $(M,g,V)$ is a static manifold with boundary, then $V$ is a \textit{Robin eigenfunction}, i.e. it satisfies \eqref{traced} in $M$ and \eqref{tracedb} on $\partial M$. 

In this paper we are mostly interested in static manifolds with boundary. They are studied recently in \cite{cruz2019prescribing,ho2020deformation,huang2022scalar,cruz2023critical,cruz2023static,sheng2024localized,medvedev2024static,sheng2025static,medvedev2025some,sheng2026obata,medvedev2026dimension} (see also references therein).

In order to state our first result in this paper, we recall the definition of \textit{simple static manifolds with boundary}, which was introduced in~\cite{sheng2025static,sheng2026obata}: $(M,g,V)$ is a simple static manifold with boundary if $M$ and its boundary are connected. 

\begin{theorem}\label{thm_main}
Let $(M^n,g,V)$, $n\geqslant3$, be a simple static manifold with boundary. If $n>3$ we additionally assume that $\partial M$ is compact. Suppose that $V$ does not vanish on $M$ and $R_g>-\dfrac{n}{n-1}H_g^2$. Then $R_g>0$, $H_g<0$, $(M,g)$ is compact, and
\[
\operatorname{Vol}_g(\partial M)=\frac{\displaystyle\int_{\partial M}R_{\partial M}\,d\sigma_g-2\int_{\partial M}|\nabla^{\partial M}\log V|_g^2\,d\sigma_g}{R_g+\dfrac{nR_g^2}{n-1}\left(\frac{\displaystyle\int_MV\,d\mu_g}{\displaystyle\int_{\partial M}V\,d\sigma_g}\right)^2}.
\]
Moreover,
\[
\operatorname{Vol}_g(M)<-\frac{H_g}{R_g}\operatorname{Vol}_g(\partial M)=\frac{\displaystyle\int_MV\,d\mu_g}{\displaystyle\int_{\partial M}V\,d\sigma_g}\operatorname{Vol}_g(\partial M).
\]
\end{theorem}

One of the main ingredients in the proof of this theorem is the following lemma.

\begin{lemma}\label{prop:positive-scalar-compactness}
Let $(M^n,g,V)$, $n\geqslant3$, be a connected static manifold with boundary. If $n>3$ we additionally assume that $\partial M$ is compact (but not necessary connected). Suppose that $R_g>0$ and $V$ does not vanish on $M$. Then $M$ is compact and at least one boundary component has negative mean curvature.
\end{lemma}

The proof of this result involves the associated Einstein manifold. We provide the requisite background in Section~\ref{sec_prelim}. This technique proves equally useful in the cases of zero and negative scalar curvature, allowing us to establish the following.

\begin{theorem}\label{prop:scalar-flat-disconnected-boundary}
Let $(M^n,g,V)$, $n\geqslant3$, be a connected static manifold with boundary. Suppose that $R_g=0$ and $V$ does not vanish on $M$. Assume that $\partial M$ is disconnected, that at least one connected component of $\partial M$ is compact, and that every boundary component has nonnegative mean curvature. Then $\partial M$ has exactly two connected components and
\[
    (M,g)\cong\left([0,a]\times\Sigma,\;dt^2+g_\Sigma\right)
\]
for some compact connected Riemannian manifold $(\Sigma,g_\Sigma)$. Moreover, $V=const$, $(\Sigma,g_\Sigma)$ and $(M,g)$ are Ricci-flat, and both boundary
components are totally geodesic. 
\end{theorem}

\begin{theorem}\label{prop:scalar-flat-noncompact-boundary}
Let $(M^n,g,V)$, $n\geqslant3$, be a connected noncompact static manifold with compact boundary. Suppose that $R_g=0$, $V$ does not vanish on $M$ and that every boundary component has nonnegative mean curvature. Then $\partial M$ is connected and
\[
    (M,g)\cong\left([0,\infty)\times\partial M,\;dt^2+g_{\partial M}\right).
\]
Moreover, $V=const$, $(M,g)$ and $(\partial M,g_{\partial M})$ are Ricci-flat and $\partial M$ is totally geodesic.
\end{theorem}

\begin{theorem}\label{prop:negative-scalar-bound}
Let $(M^n,g,V)$, $n\geqslant3$, be a connected noncompact static manifold with compact boundary. Suppose that $R_g<0$, $H_g>0$, and $V$ does not vanish on $M$. Then 
\[
   R_g\leqslant-\frac{n}{n-1} H_g^2.
\]

If equality holds, then, setting
\[
    \kappa=\sqrt{-\frac{R_g}{n(n-1)}},
\]
one has $H_g=(n-1)\kappa$, and
\[
    (M,g)\cong\left([0,\infty)\times\partial M,\;dt^2+e^{-2\kappa t}g_{\partial M}\right).
\]
Moreover, $V(t,x)=e^{-\kappa t}V_0(x)$, where $V_0=V|_{\partial M}$.
\end{theorem} 

\subsection{Comparison with Sheng--Zhao's theorem}

Theorems~\ref{prop:scalar-flat-disconnected-boundary}--\ref{prop:negative-scalar-bound} and Theorem~1.4 of~\cite{sheng2026obata} are complementary.

In the scalar-flat noncompact case with compact boundary and $H_g=0$, Sheng--Zhao's theorem yields that $(M,g)$ is Ricci-flat, $\partial M$ is totally geodesic, and $V$ is constant. Under the assumptions of Theorem~\ref{prop:scalar-flat-noncompact-boundary}, we obtain in addition a global splitting. No decay assumption at infinity is required. Theorem~\ref{prop:scalar-flat-disconnected-boundary} also treats disconnected boundary, which is excluded by the simplicity assumption in~\cite{sheng2026obata}. 

For $R_g<0$, Theorem~\ref{prop:negative-scalar-bound} gives, for a noncompact simple static manifold with compact boundary, $R_g\leqslant-\dfrac{n}{n-1}H_g^2$. Hence, combined with the opposite inequality assumed in Sheng--Zhao's theorem, $R_g\geqslant-\dfrac{n}{n-1}H_g^2$ equality is forced. In this case one recovers the same exponential warped-product rigidity\footnote{Sheng--Zhao use the coordinate $t\in(-\infty,0]$ increasing toward the boundary, whereas here we use $t\in[0,\infty)$ increasing inward from the boundary. Thus their factor $e^{2H_gt/(n-1)}$ becomes $e^{-2H_gt/(n-1)}$ after the change of variable $t\mapsto -t$.}. In this regime, Theorem~\ref{prop:negative-scalar-bound} does not require the decay condition imposed in the noncompact case in~\cite{sheng2026obata}, nor the additional upper bound $R_g\leqslant-H_g^2$.

On the other hand, Theorem~1.4 of~\cite{sheng2026obata} applies in several situations not covered by our results, including compact manifolds with connected boundary and $H_g\geqslant 0$. Moreover, under the assumption
\[
-\frac{n}{n-1}H_g^2\leqslant R_g\leqslant-H_g^2,
\]
Sheng--Zhao prove the Obata-type equation
\[
\operatorname{Hess}_gV+\frac{R_g+H_g^2}{n-1}Vg=0 \quad \text{in} \quad M
\]
together with
\[
\frac{\partial V}{\partial\nu}=\frac{H_g}{n-1}V \quad \text{on} \quad \partial M,
\]
and in particular show that $(M,g)$ is Einstein.

Thus, in the scalar-flat setting and in the noncompact negative-scalar-curvature regime, Theorems~\ref{prop:scalar-flat-disconnected-boundary}--\ref{prop:negative-scalar-bound} provide additional global conclusions or require weaker auxiliary assumptions, while Sheng--Zhao's theorem covers other regimes and yields the stronger Obata-type rigidity under its curvature hypotheses.

\subsection*{Paper organization} The proof of Theorem~\ref{thm_main}, together with the proofs of all the results stated in the Introduction, is given in Section~\ref{sec_proofs}. Section~\ref{sec_notation} collects the notation and conventions used throughout the paper, while Section~\ref{sec_prelim} contains the necessary technical background. In Section~\ref{appendix}, we derive several simple but apparently new identities for static manifolds with boundary. There we also discuss static manifolds without boundary and illustrate how the associated Einstein manifold technique applies in that setting.

\subsection*{Acknowledgments} The author thanks Lucas Ambrozio for sharing the argument in the proof of Theorem~\ref{thm:positive-scalar-curvature}, which inspired this article.  During the preparation of this manuscript, OpenAI's ChatGPT was utilized to assist with structuring the text, refining the exposition, proofreading, and offering feedback on the presentation of mathematical arguments. The author carried out all computations and proofs by hand and independently reviewed and verified all AI-generated suggestions. The author takes full responsibility for the mathematical content and conclusions of the article. This article is an output of a research project (HSE-BR-2025-059) implemented as part of the Basic Research Program at the National Research University Higher School of Economics (HSE University).

\section{Notation and conventions}\label{sec_notation}

Throughout this article, we use the following notation and conventions.

\begin{itemize}
\item $\operatorname{Ric}_g$ and $R_g$ denote the Ricci tensor and scalar curvature of a Riemannian manifold $(M,g)$, respectively.

\item $B_g$ and $H_g$ denote the second fundamental form and mean curvature of $\partial M$ with respect to the outward unit normal $\nu$. We use the conventions
\[
B_g(X,Y)=g(\nabla^g_X\nu,Y),\qquad H_g=\operatorname{tr}_{g_{\partial M}}B_g,
\]
so that Euclidean spheres have positive mean curvature. Writing $\partial M$ as the disjoint union of its connected components, we set
\[
\partial M=\bigsqcup_{i=1}^k\partial_iM,\qquad B_i:=B_g|_{\partial_iM},\qquad H_i:=H_g|_{\partial_iM}.
\]

\item $\operatorname{Hess}_g$ denotes the Hessian, $\Delta_g=\operatorname{tr}_g\operatorname{Hess}_g$ denotes the Laplacian, and $\nabla^g$ denotes the gradient.

\item When no separate notation is introduced for the induced metric on a submanifold $\Sigma$, we write $R_\Sigma$, $\nabla^\Sigma$, and $\Delta_\Sigma$ for its scalar curvature, gradient, and Laplacian. If $\Sigma$ is a hypersurface with a chosen unit normal, $B_\Sigma$ and $H_\Sigma$ denote its second fundamental form and mean curvature, with the same conventions as above.

\item $\operatorname{Vol}_g(M)$ denotes the Riemannian volume of $(M,g)$. When $M$ is two-dimensional, we also write $\operatorname{Area}_g(M)$ to emphasize its dimension.

\item $\mathcal P$ denotes the vector space of static potentials on a given Riemannian manifold with boundary $(M^n,g)$:
\[
\begin{split} \mathcal P=\biggl\{f\in C^\infty(M):\;&\operatorname{Hess}_g f=f\left(\operatorname{Ric}_g-\frac{R_g}{n-1}g\right)\ \text{in }M,\\ &(\partial_\nu f)g_{\partial M}=fB_g\ \text{on }\partial M\biggr\}. \end{split}
\]
\end{itemize}

\section{Preliminaries}\label{sec_prelim}

For the reader's convenience, we collect in this section some basic properties of static manifolds with boundary

\subsection{Identities} There are several identities for static manifolds with boundary established in the literature. We start probably with the simplest one. Assuming that $(M,g)$ is compact, integrating~\eqref{traced} and using~\eqref{tracedb}, we obtain
\[
-\frac{R_g}{n-1}\int_MV\,d\mu_g=\int_M\Delta_gV\,d\mu_g=\sum_{j=1}^k\int_{\partial_iM}\partial_\nu V\,d\sigma_g=\frac{1}{n-1}\sum_{i=1}^kH_i\int_{\partial_iM}V\,d\sigma_g.
\]
Therefore,
\begin{align}\label{static_ident}
-R_g\int_MV\,d\mu_g=\sum_{i=1}^kH_i\int_{\partial_iM}V\,d\sigma_g.
\end{align}

This, in particular, implies that if $(M,g,V)$ is simple and $V$ does not vanish on $M$, then \eqref{static_ident} implies that $R_g$ and $H_g$ are of opposite sign and $R_g=0$ if and only if $H_g=0$ (for the case of a static manifold with boundary of general type, see e.g. Theorem 3.5 in~\cite{sheng2025static}).

Further, in~\cite{medvedev2024static} we established a general formula (1.6), which for a static manifold with boundary $(M^n,g,V)$ with nowhere-vanishing $V$ yields
\begin{align}\label{med}
\int_MV\left|\mathring{\operatorname{Ric}}_g\right|_g^2\,d\mu_g+\sum_{i=1}^kH_i\left(\left(\frac{H_i}{n-1}\right)^2+\frac{R_g}{n(n-1)}\right)\int_{\partial_iM}V\,d\sigma_g=0.
\end{align}
This formula can be also deduced from Reilly's formula applied to $f=\sqrt{V}$ on $(M,g)$ or from the Gauss--Bonnet--Chern formula on the associated Einstein manifold (when $n=3$). Also notice that for a connected boundary, the formula in Proposition~2.1 of~\cite{sheng2026obata}
\begin{align}\label{sheng_zhao}
\int_M V^{-1}\left|\operatorname{Hess}_gV+\frac{R_g+H_g^2}{n-1}Vg\right|_g^2\,&d\mu_g\\\nonumber&=\frac{R_g+H_g^2}{n-1}\left(R_g+\frac{n}{n-1}H_g^2\right)\int_MV\,d\mu_g.
\end{align}
is algebraically equivalent to~\eqref{med}. Indeed, using the static equation~\eqref{static},
\[
\operatorname{Hess}_gV+\frac{R_g+H_g^2}{n-1}Vg=V\left(\mathring{\operatorname{Ric}}_g+\left(\frac{R_g}{n}+\frac{H_g^2}{n-1}\right)g\right).
\]
Since $\mathring{\operatorname{Ric}}_g$ is trace-free,
\[
V^{-1}\left|\operatorname{Hess}_gV+\frac{R_g+H_g^2}{n-1}Vg\right|_g^2=V\left|\mathring{\operatorname{Ric}}_g\right|_g^2+nV\left(\frac{R_g}{n}+\frac{H_g^2}{n-1}\right)^2.
\]
Substituting this into~\eqref{sheng_zhao} and simplifying gives
\[
\int_MV\left|\mathring{\operatorname{Ric}}_g\right|_g^2\,d\mu_g=R_g\left(\frac{H_g^2}{(n-1)^2}+\frac{R_g}{n(n-1)}\right)\int_MV\,d\mu_g.
\]
Finally, using~\eqref{static_ident} we obtain
\[
\int_MV\left|\mathring{\operatorname{Ric}}_g\right|_g^2\,d\mu_g+H_g\left(\left(\frac{H_g}{n-1}\right)^2+\frac{R_g}{n(n-1)}\right)\int_{\partial M}V\,d\sigma_g=0.
\]

We establish some new identities for static manifolds with boundary in Subsection~\ref{identities} of Section~\ref{appendix}.

\begin{remark}
The \textit{Wang mass} formula in Remark~1.16 of~\cite{medvedev2024static} and the observation in Remark~3.3 of~\cite{sheng2026obata} are essentially two formulations of the same asymptotic identity. Indeed, in the asymptotically locally hyperbolic setting, the boundary integral at infinity in~\cite{sheng2026obata} \(\displaystyle\lim_{r\to\infty}\int_{\partial M_r}\mathring{\operatorname{Ric}}_g(\nabla^gV,\nu)\,d\sigma_g\) is, up to the normalization conventions, proportional to the \textit{Wang--Chru\'sciel--Herzlich mass}. Remark~1.16 of~\cite{medvedev2024static} expresses this asymptotic quantity in terms of an integral over $M$ together with the contributions from the boundary and the zero-level components of $V$. Thus the two formulas encode the same mass invariant, with the apparent difference arising from the normalization and from the fact that~\cite{medvedev2024static} keeps the interior and boundary terms explicitly.
\end{remark}

\subsection{Associated Einstein manifolds}\label{AEM}

An important feature of static manifolds is their close relation to Einstein manifolds. In this section we recall the basic facts needed below, following~\cite{ambrozio2017static,cruz2023static}.

Let $(M^n,g,V)$ be a static manifold, and denote its zero-level set by $\Sigma:=V^{-1}(0)$. On every connected component $\Omega$ of $M\setminus\Sigma$, the potential $V$ has a fixed sign, and one considers the warped product metric $h=g+V^2d\theta^2$ on $\Omega\times S^1$. The static equation~\eqref{static} implies that $\operatorname{Ric}_h=\dfrac{R_g}{n-1}h$ so that $(\Omega\times S^1,h)$ is Einstein away from the zero set of $V$. 

Since the length of the circle-generating Killing field $X=\partial_\theta$ is $|X|_h=|V|$, the $S^1$-orbits collapse as one approaches $\Sigma$. The \textit{associated singular Einstein manifold} $(N^{n+1},h)$ of a static manifold $(M,g,V)$ is obtained by adjoining the zero set and collapsing each circle fiber over $\Sigma$ to a point. Thus, $\Sigma$ becomes a fixed-point set of the natural isometric $S^1$-action on $N$. The metric $h$ extends smoothly across a connected component $\Sigma_i$ of $\Sigma$ when the potential is normalized so that $|\nabla^gV|_g|_{\Sigma_i}=1$. Otherwise, the associated Einstein metric has a conical singularity along the fixed-point set. In this paper we are mostly interested in the situation when $V$ does not vanish on $M$. In this case the metric $h$ has no singularities at all and $(N,h)$ is a genuine (nonsingular) Einstein manifold

If $(M,g,V)$ is a static manifold with boundary, then $\partial N=\partial M\times S^1$ and $(N,h)$ is the associated Einstein manifold has boundary. This boundary is umbilical. Indeed,~\eqref{staticb} and~\eqref{tracedb} imply that the second fundamental form and mean curvature of a connected component $\partial_iN=\partial_iM\times S^1$ are 
\begin{align}\label{aemBH}
B_{\partial_iN}=\frac{H_i}{n-1}h_{\partial_iN}, \qquad H_{\partial_iN}=\frac{n}{n-1}H_i.
\end{align}

\section{Proofs}\label{sec_proofs}

We begin by proving the main ingredient needed for Theorem~\ref{thm_main}.

\begin{proof}[Proof of Lemma~\ref{prop:positive-scalar-compactness}]
Suppose, by contradiction, that $M$ is noncompact. Consider the associated Einstein manifold $(N,h)$. Since $V$ does not vanish on $M$, $(N,h)$ is nonsingular, complete, and noncompact. Set $\kappa=\dfrac{R_g}{n(n-1)}>0$. We have $\operatorname{Ric}_h=n\kappa h$.

The boundary of $N$ is $\partial N=\partial M\times S^1$, and hence it is compact. Consequently, there exists a constant $A\in\mathbb R$ such that the mean curvature on $\partial N$ satisfies $H_{\partial N}\geqslant nA$. 

We have: $(N,h)$ is connected, complete, has compact boundary, and belongs to the class $(\kappa,A)$ in terms of~\cite{kasue1983ricci}\footnote{The sign convention for the boundary mean curvature in~\cite{kasue1983ricci} is opposite to the one adopted in this paper.}. Then \cite[Theorem~C(1)]{kasue1983ricci} implies that $\kappa\leqslant0$. This contradicts $\kappa>0$. Hence $M$ is compact.

The conclusion that at least one boundary component of $M$ has negative mean curvature follows directly from~\eqref{static_ident}.

In the case where $\dim M=3$, one can prove that $\partial M$ is automatically compact. Indeed, it is not difficult to verify that the restriction of $V$ to any boundary component $\partial_iM$ is an eigenfunction of the Jacobi operator with eigenvalue 0:
\begin{align*}
J_{\partial_iM}V&=-\Delta_{\partial_iM}V-\left(\operatorname{Ric}_g(\nu,\nu)+|B_{\partial_iM}|_g^2\right)V\\&=-\Delta_{\partial_iM}V+\frac12R_{\partial_iM}V-\frac12\left(R_g+\frac{n}{n-1}H_i^2\right)V=0.
\end{align*}
Since $V$ does not vanish on $\partial_iM$, $V$ is a first eigenfunction. This implies that $\partial M$ is a stable (possibly disconnected) CMC surface in $(M,g)$. Since $R_g>0$ is constant, $(M,g)$ satisfies the assumptions of Theorem 1 in~\cite{rosenberg2006constant}\footnote{The corresponding compactness statement is special to dimension three. In higher dimensions it fails even under strong stability and uniformly positive scalar curvature. For example, in $\mathbb S^2\times\mathbb R^{n-2}$ with product metric the totally geodesic hypersurface $\mathbb S^2\times\mathbb R^{n-3}\times\{0\}$ is complete, non-compact, and stable, while the ambient scalar curvature is identically equal to $2$.}. Hence, $\partial M$ is compact.
\end{proof}

Since Theorems~\ref{prop:scalar-flat-disconnected-boundary},~\ref{prop:scalar-flat-noncompact-boundary}, and~\ref{prop:negative-scalar-bound} rely on the same main idea as the previous proof -- namely, the passage to the associated Einstein manifold -- we prove them before the main theorem.

\begin{proof}[Proof of Theorem~\ref{prop:scalar-flat-disconnected-boundary}]
The associated Einstein $(N,h)$ is complete and Ricci-flat. By \eqref{aemBH}, then the mean curvature of a connected component $\partial_iN$ of $\partial N$
\[
    H_{\partial_iN}=\frac{n}{n-1}H_i\geqslant0.
\]
Since $\partial N$ is disconnected and has a compact connected component, Theorem~B in~\cite{kasue1983ricci} implies that
\[
    (N,h)\cong\left([0,a]\times\Gamma,\;dt^2+h_\Gamma\right),
\]
where $\Gamma$ is a compact connected component of $\partial N$. In particular, $\partial N$ has exactly two connected components.

The $S^1$-action preserves every connected component of $\partial N$. Moreover, the distance function from $\Gamma$ is invariant under the
$S^1$-action. Hence the above product splitting is $S^1$-equivariant. Passing to the quotient by the circle action gives
\[
    (M,g)\cong\left([0,a]\times\Sigma,\;dt^2+g_\Sigma\right),
\]
where $ \Sigma=\Gamma/S^1$. In particular, $(M,g)$ is compact and the boundary components are totally geodesic. Since $R_g=0$ equations~\eqref{traced} and \eqref{tracedb} then imply that $V=const$, as $\Delta_gV=0$ in $M$ and $\partial_\nu V=0$ on $\partial M$.

The static equation~\eqref{static} implies that $\operatorname{Ric}_g=0$. For the product metric $g=dt^2+g_\Sigma$ the Ricci tensor satisfies
\[
    \operatorname{Ric}_g|_{T\Sigma}=\operatorname{Ric}_{g_\Sigma}.
\]
Therefore $\operatorname{Ric}_{g_\Sigma}=0$.
\end{proof}

\begin{proof}[Proof of Theorem~\ref{prop:scalar-flat-noncompact-boundary}]
 As in the proof of Proposition~\ref{prop:scalar-flat-disconnected-boundary}, the associated Einstein manifold $(N,h)$ is complete, Ricci-flat and $H_{\partial N}\geqslant0$. Then Theorem~C in \cite{kasue1983ricci} implies that $\partial N$ is connected and that
\[
    (N,h)\cong\left([0,\infty)\times\partial N,\;dt^2+h_{\partial N}\right).
\]
Since $\partial N=\partial M\times S^1$ it follows that $\partial M$ is connected.

The distance from $\partial N$ is invariant under the circle action, so the splitting is $S^1$-equivariant. Passing to the orbit space gives
\[
    (M,g)\cong\left([0,\infty)\times\partial M,\;dt^2+g_{\partial M}\right).
\]

Let $\{\Phi_s\}_{s\in S^1}$ denote the $S^1$-action on $N$, and let $X$ be its generating Killing field, so that $X_p=\left.\frac{d}{ds}\right|_{s=0}\Phi_s(p)$. The Killing field $X$ is tangent to each slice $\{t\}\times\partial N$ and is independent of the $t$-variable. Indeed, since the $S^1$-action preserves the boundary component $\partial N$ and acts by isometries, it preserves the distance function $t(p)=d_h(p,\partial N)$. Hence $t\circ\Phi_s=t$ for every $s$, and differentiating at $s=0$ gives $X(t)=0$. Thus $X$ is tangent to every slice $\{t\}\times\partial N$. Moreover, the product splitting is given by the normal exponential map $F(t,x)=\exp_x(t\nu_x)$, where $\nu$ is the inward unit normal along $\partial N$. Since $\Phi_s$ preserves both $\partial N$ and its inward unit normal, and isometries commute with the exponential map, $\Phi_s(F(t,x))=F\left(t,\Phi_s|_{\partial N}(x)\right)$. Therefore, in the product coordinates, $\Phi_s(t,x)=\left(t,\Phi_s|_{\partial N}(x)\right)$, so the action is independent of $t$. Consequently, $X_{(t,x)}=(0,X^{\partial N}_x)$, where $X^{\partial N}$ is the restriction of $X$ to $\partial N$, and hence $|X|_h(t,x)=|X^{\partial N}|_{h_{\partial N}}(x)$. Since $V=|X|_h$, it follows that $V(t,x)=V_0(x)$, where $V_0:=V|_{\partial N}=|X^{\partial M}|_{h_{\partial N}}$. 

Since $R_g=0$,~\eqref{traced} yields $\Delta_gV=0$. Because $V$ is independent of $t$, this reduces to $\Delta_{g_{\partial M}}V_0=0$. As $\partial M$ is compact, $V_0$ is constant.

The boundary of the Riemannian product $[0,\infty)\times\partial M$ is totally geodesic.

Finally, arguing as in the end of the proof of Proposition~\ref{prop:scalar-flat-disconnected-boundary}, we conclude that $M$ and $\partial M$ are Ricci-flat.
\end{proof}

\begin{remark}

\begin{enumerate}[(i)]
\item Notice that the upper half-space and a slab between two parallel hyperplanes in Euclidean $n$-space with constant functions as static potential provide examples of noncompact static manifolds with noncompact boundary.
\item Consider compact scalar-flat static manifolds with \textit{connected} boundary and nowhere-vanishing static potentials. By~\eqref{traced} and~\eqref{tracedb}, any static potential in this case is a first Steklov eigenfunction, hence constant. Then $H_g=0$ and the boundary is totally geodesic. We provide an example of an orientable such static manifold.
\begin{example}
Let $\widetilde M=[-1,1]\times T^{n-1}$ be equipped with the flat product metric $\widetilde g=dt^2+g_{T^{n-1}}$. Choose a fixed-point-free isometric involution $a:T^{n-1}\longrightarrow T^{n-1}$ for example, translation by half of a lattice vector, and define $\iota(t,x)=(-t,a(x))$. The map $\iota$ is a free isometric involution of $(\widetilde M,\widetilde g)$. Consider the quotient
\[
(M,g):=(\widetilde M,\widetilde g)/\langle\iota\rangle.
\]

The two boundary components $\{-1\}\times T^{n-1}$ and $\{1\}\times T^{n-1}$are interchanged by $\iota$. Therefore, the quotient manifold $M$ has connected boundary. Notice that for $n\geqslant3$ $M$ is orientable.

Since $\widetilde g$ is flat and the quotient map is a local isometry, the metric $g$ is flat. Hence $\operatorname{Ric}_g=0$, $R_g=0$. Moreover, the slices $\{\pm1\}\times T^{n-1}$ are totally geodesic in the product manifold, and therefore $B_g=0$, $H_g=0$ on $\partial M$.

It is easy to verify that~\eqref{static} and \eqref{staticb} are satisfied for $V=const$.
\end{example}
\end{enumerate}
\end{remark}

\begin{proof}[Proof of Theorem~\ref{prop:negative-scalar-bound}]
Consider the associated Einstein manifold $(N,h)$. Then
\[
    \operatorname{Ric}_h=\frac{R_g}{n-1}h=-n\kappa^2h.
\]
Also, by~\eqref{aemBH}, $H_{\partial N}=-nA$, where $A=-\dfrac{H_g}{n-1}$. Since $(N,h)$ is complete and noncompact, with compact boundary, and since $-\kappa^2<0$ and $A<0$, Theorem~C from~\cite{kasue1983ricci} gives $A\geqslant-\kappa$. Whence $\dfrac{H_g}{n-1}\leqslant\kappa$. Squaring both sides gives
\[
    H_g^2\leqslant(n-1)^2\kappa^2=-\frac{n-1}{n}R_g.
\]

Suppose now that equality holds. Then $H_g=(n-1)\kappa$ and hence $A=-\kappa$. The rigidity statement in Theorem~C in \cite{kasue1983ricci} yields
\[
    (N,h)\cong\left([0,\infty)\times\partial N,\;dt^2+w(t)^2h_{\partial N}\right),
\]
where $w$ is the solution of
\[
    w''-\kappa^2w=0,\qquad w(0)=1,\qquad w'(0)=-\kappa.
\]
Therefore $w(t)=e^{-\kappa t}$ and $h=dt^2+e^{-2\kappa t}h_{\partial N}$

The distance from $\partial N$ is invariant under the $S^1$-action, so the warped-product decomposition is $S^1$-equivariant. Passing to the quotient by the circle action gives $g= dt^2+e^{-2\kappa t}g_{\partial M}$.

Let $X$ be the Killing field generating the $S^1$-action. On $\partial N$ one has
\[
    |X|_{h}=V_0, \qquad V_0=V|_{\partial M}.
\]
Since $h=dt^2+e^{-2\kappa t}h_{\partial N}$ the length of $X$ at $(t,x)$ is
\[
    |X|_h=e^{-\kappa t}|X|_{h_{\partial N}}=e^{-\kappa t}V_0(x).
\]
Since $|X|_h=V$, it follows that $V(t,x)=e^{-\kappa t}V_0(x)$.

Finally, the Ricci tensor of the warped product $h=dt^2+e^{-2\kappa t}h_{\partial N}$ in directions tangent to $\partial N$ is
\[
    \operatorname{Ric}_h|_{T\partial N}=\operatorname{Ric}_{h_{\partial N}}-n\kappa^2e^{-2\kappa t}h_{\partial N}.
\]
On the other hand, $\operatorname{Ric}_h=-n\kappa^2h$, and therefore
\[
    \operatorname{Ric}_h|_{T\partial N}=-n\kappa^2e^{-2\kappa t}h_{\partial N}.
\]
Comparing the two identities gives $\operatorname{Ric}_{h_{\partial N}}=0$. Since $h_{\partial N}=g_{\partial M}+V_0^2d\theta^2$ the associated metric on $\partial M\times S^1$ is Ricci-flat.
\end{proof}

\begin{example}
Let $(\Sigma^{n-1},g_\Sigma)$ be a closed flat manifold, for example a flat torus $T^{n-1}$, and let $a<b$. Consider
\[
M=[a,b]\times\Sigma,\qquad g=dt^2+e^{2t}g_\Sigma,\qquad V=e^t.
\]
Then $M$ is compact, $\partial M$ has two connected components, and $V$ is strictly positive on $M$. One boundary component has mean curvature $n-1$ and the other one $-(n-1)$.

Since $(\Sigma,g_\Sigma)$ is flat, the warped product metric $g$ has constant sectional curvature $-1$. Consequently,
\[
\operatorname{Ric}_g=-(n-1)g,\qquad R_g=-n(n-1)<0.
\]

It is easy to verify that $V$ satisfies both~\eqref{static} and~\eqref{staticb}. Hence $(M,g,V)$ is a compact static manifold with two boundary components with negative scalar curvature and a nowhere-vanishing static potential.
\end{example}



Now we come back to the proof of Theorem~\ref{thm_main}. The second ingredient in it is a corollary from the following lemma.

\begin{lemma}\label{lem_ident}
Let $(M^n,g,V)$, $n\geqslant3$, be a static manifold with boundary such that $V$ does not vanish on a compact boundary component $\partial_iM$. Then
\begin{align}\label{ident_lem}
\int_{\partial_iM}R_\gamma\,d\sigma_g=\operatorname{Vol}_g(\partial_iM)\left(R_g+\frac{n}{n-1}H_i^2\right)+2\int_{\partial_iM}|\nabla^\gamma\log|V||_\gamma^2\,d\sigma_g,
\end{align}
where $\gamma$ is the induced metric on $\partial_iM$.
In particular, when $n=3$,
\[
2\pi\chi(\partial_iM)=\operatorname{Area}_g(\partial_iM)\left(\frac{R_g}{2}+\frac{3}{4}H_i^2\right)+\int_{\partial_iM}|\nabla^\gamma\log|V||_\gamma^2\,d\sigma_g.
\]
\end{lemma}

\begin{proof}
The induced metric on the boundary component $\partial_iN$ corresponding to $\partial_iM$ is
\[
h_{\partial_iN}=\gamma+V^2d\theta^2.
\]
Since the fibre is one-dimensional, the scalar curvature of this warped product is
\[
R_{\partial_iN}=R_\gamma-2\frac{\Delta_\gamma V}{V}.
\]

Using the contracted Gauss equation, umbilicity of $\partial_iN$ in $(N,h)$ and formulas from Section~\ref{AEM}, we obtain
\[
R_{\partial_iN}=R_g+\frac{n}{n-1}H_i^2.
\]
Combining this with the warped-product scalar-curvature formula yields
\[
R_\gamma=\left(R_g+\frac{n}{n-1}H_i^2\right)+2\frac{\Delta_\gamma V}{V}.
\]

Since $V$ does not vanish on the connected component $\partial_iM$, it has a fixed sign there, and hence
\[
\nabla^\gamma\log|V|=\frac{\nabla^\gamma V}{V}.
\]
Integrating over $\partial_iM$ gives
\[
\int_{\partial_iM}R_\gamma\,d\sigma_g=\operatorname{Area}_g(\partial_iM)\left(R_g+\frac{n}{n-1}H_i^2\right)+2\int_{\partial_iM}\frac{\Delta_\gamma V}{V}\,d\sigma_g.
\]
Since $\partial_iM$ is closed,
\[
0=\int_{\partial_iM}\operatorname{div}_\gamma\left(\frac{\nabla^\gamma V}{V}\right)d\sigma_g=\int_{\partial_iM}\left(\frac{\Delta_\gamma V}{V}-\frac{|\nabla^\gamma V|_\gamma^2}{V^2}\right)d\sigma_g.
\]
Thus
\[
\int_{\partial_iM}\frac{\Delta_\gamma V}{V}\,d\sigma_g=\int_{\partial_iM}|\nabla^\gamma\log|V||_\gamma^2\,d\sigma_g.
\]
Substituting this identity gives
\[
\int_{\partial_iM}R_\gamma\,d\sigma_g=\operatorname{Area}_g(\partial_iM)\left(R_g+\frac{n}{n-1}H_i^2\right)+2\int_{\partial_iM}|\nabla^\gamma\log|V||_\gamma^2\,d\sigma_g.
\]

When $n=3$, one has $R_\gamma=2K_\gamma$. Hence, applying the Gauss--Bonnet theorem, we obtain
\[
2\pi\chi(\partial_iM)=\operatorname{Area}_g(\partial_iM)\left(\frac{R_g}{2}+\frac{3}{4}H_i^2\right)+\int_{\partial_iM}|\nabla^\gamma\log|V||_\gamma^2\,d\sigma_g.
\]
\end{proof}

\begin{remark}\label{rem_jacobi}
Lemma~\ref{lem_ident} implies
\begin{align}\label{ineq_stab}
\int_{\partial_iM}R_\gamma\,d\sigma_g\geqslant\operatorname{Vol}_g(\partial_iM)\left(R_g+\frac{n}{n-1}H_i^2\right),
\end{align}
with equality if and only if $V$ is constant on $\partial_iM$.

It is worth noting that inequality~\eqref{ineq_stab} can also be obtained from the fact that, if $V$ does not vanish on $\partial_iM$, then $V|_{\partial_iM}$ is a first eigenfunction of the Jacobi operator of $\partial_iM$ with eigenvalue $0$; see the argument at the end of the proof of Lemma~\ref{prop:positive-scalar-compactness}. Lemma~\ref{lem_ident}, however, is stronger, since it yields the corresponding identity.

It is also worth noting that, under the hypotheses of the following proposition, the volumes of $M$ and $\partial M$ admit upper bounds depending only on $R_g$.

\begin{proposition}\label{vol_area}
Let $(M^n,g,V)$, $n\geqslant3$, be a compact simple static manifold with boundary. Suppose that $R_g>0$ and $V$ does not vanish on $M$. Then 
\[ 
\operatorname{Vol}_g(M)<\frac{-H_g}{R_g\left(R_g+\dfrac{n}{n-1}H_g^2\right)}\int_{\partial M}R_\gamma\,d\sigma_g\leqslant\frac{1}{2R_g^{3/2}}\sqrt{\frac{n-1}{n}}\int_{\partial M}R_\gamma\,d\sigma_g 
\] 
and 
\[ 
\operatorname{Vol}_g(\partial M)\leqslant\frac{1}{R_g+\dfrac{n}{n-1}H_g^2}\int_{\partial M}R_\gamma\,d\sigma_g<\frac{1}{R_g}\int_{\partial M}R_\gamma\,d\sigma_g. 
\] 
In particular, when $n=3$ and $\partial M$ is orientable
\[ 
\operatorname{Vol}_g(M)<\frac{8\pi(-H_g)}{R_g\left(R_g+\dfrac32H_g^2\right)}\leqslant\frac{8\pi}{\sqrt6\,R_g^{3/2}} \quad \text{and}\quad \operatorname{Area}_g(\partial M)\leqslant\frac{8\pi}{R_g+\dfrac32H_g^2}<\frac{8\pi}{R_g}.
\] 
\end{proposition}

\begin{proof}
By~\eqref{lem_ident}, 
\[ 
\operatorname{Vol}_g(\partial M)\leqslant\frac{1}{R_g+\dfrac{n}{n-1}H_g^2}\int_{\partial M}R_\gamma\,d\sigma_g<\frac{1}{R_g}\int_{\partial M}R_\gamma\,d\sigma_g. 
\] 
On the other hand, by~\eqref{traced} and~\eqref{tracedb}, $V$ is the first eigenfunction of 
\[ 
\begin{cases}
 -\Delta_gV-\dfrac{R_g}{n-1}V=0 & \text{in }M,\\ 
\partial_\nu V=\dfrac{H_g}{n-1}V & \text{on }\partial M. 
\end{cases} 
\] 
Hence, by the variational characterization, 
\[ 
\frac{H_g}{n-1}=\inf_{\phi|_{\partial M}\not\equiv0}\frac{\displaystyle\int_M|\nabla^g\phi|_g^2\,d\mu_g-\dfrac{R_g}{n-1}\int_M\phi^2\,d\mu_g}{\displaystyle\int_{\partial M}\phi^2\,d\sigma_g}. 
\] 
Testing with $\phi\equiv1$ gives, strictly since $R_g>0$,
\[ 
\frac{\operatorname{Vol}_g(M)}{\operatorname{Vol}_g(\partial M)}<-\frac{H_g}{R_g}. 
\] 
Therefore 
\[ 
\operatorname{Vol}_g(M)<\frac{-H_g}{R_g\left(R_g+\dfrac{n}{n-1}H_g^2\right)}\int_{\partial M}R_\gamma\,d\sigma_g. 
\] 
Since 
\[ 
\frac{x}{R_g\left(R_g+\dfrac{n}{n-1}x^2\right)}\leqslant\frac{1}{2R_g^{3/2}}\sqrt{\frac{n-1}{n}},\qquad x\geqslant0, 
\] 
the second volume estimate follows. 

When $n=3$, the boundary identity and $R_g>0$ imply $\chi(\partial M)>0$. If $\partial M$ is connected and orientable, \( \partial M\cong S^2, \) and hence 
\[ 
\int_{\partial M}R_\gamma\,d\sigma_g=8\pi. 
\] 
Substituting this into the preceding estimates gives the stated three-dimensional inequalities.
\end{proof}
\end{remark}

Lemma~\ref{lem_ident} admits the following corollary, which is another ingredient in the proof of Theorem~\ref{thm_main}.

\begin{corollary}\label{thm_simple}
Let $(M^n,g,V)$, $n\geqslant3$, be a compact static manifold with connected boundary such that $V$ does not vanish on $M$. One has
\[
\int_{\partial M}R_\gamma\,d\sigma_g=\left[R_g+\frac{nR_g^2}{n-1}\left(\frac{\displaystyle\int_MV\,d\mu_g}{\displaystyle\int_{\partial M}V\,d\sigma_g}\right)^2\right]\operatorname{Vol}_g(\partial M)+2\int_{\partial M}|\nabla^\gamma\log V|_\gamma^2\,d\sigma_g.
\]
In particular, when $n=3$,
\[
2\pi\chi(\partial M)=\left[\frac{R_g}{2}+\frac{3R_g^2}{4}\left(\frac{\displaystyle\int_MV\,d\mu_g}{\displaystyle\int_{\partial M}V\,d\sigma_g}\right)^2\right]\operatorname{Area}_g(\partial M)+\int_{\partial M}|\nabla^\gamma\log V|_\gamma^2\,d\sigma_g.
\]
\end{corollary}

\begin{proof}
By Lemma~\ref{lem_ident},
$$
\int_{\partial M}R_\gamma\,d\sigma_g=\left(R_g+\frac{n}{n-1}H_g^2\right)\operatorname{Area}_g(\partial M)+2\int_{\partial M}|\nabla^\gamma\log V|_\gamma^2\,d\sigma_g.
$$
Identity~\eqref{static_ident} gives
$$
H_g=-R_g\frac{\displaystyle\int_MV\,d\mu_g}{\displaystyle\int_{\partial M}V\,d\sigma_g}.
$$
Substituting this into the previous identity yields
$$
\int_{\partial M}R_\gamma\,d\sigma_g=\left[R_g+\frac{nR_g^2}{n-1}\left(\frac{\displaystyle\int_MV\,d\mu_g}{\displaystyle\int_{\partial M}V\,d\sigma_g}\right)^2\right]\operatorname{Area}_g(\partial M)+2\int_{\partial M}|\nabla^\gamma\log V|_\gamma^2\,d\sigma_g.
$$
\end{proof}

Now we ready to prove Theorem~\ref{thm_main}.

\begin{proof}[Proof of Theorem~\ref{thm_main}]

First, we show that $R_g>0$.  Replacing $V$ by $-V$ if necessary, we may assume that $V>0$ on $M$. Using identity~\eqref{static_ident} and the connectedness of $\partial M$, identity~\eqref{med} can be written as
\[
\int_MV|\mathring{\operatorname{Ric}}_g|_g^2\,d\mu_g=R_g\left(\frac{H_g^2}{(n-1)^2}+\frac{R_g}{n(n-1)}\right)\int_MV\,d\mu_g.
\]
By the assumption $R_g>-\dfrac{n}{n-1}H_g^2$, the factor in parentheses is positive. Since the left-hand side is nonnegative, it follows that $R_g\geqslant0$. If $R_g=0$, then identity~\eqref{static_ident} gives $H_g=0$, contradicting the strict inequality above. Hence $R_g>0$.

By Theorem~\ref{prop:positive-scalar-compactness}, $(M,g)$ is compact. Since $(M,g,V)$ is simple and $V>0$, identity~\eqref{static_ident} then implies $H_g<0$. The remaining assertions follow directly from Corollary~\ref{thm_simple}, while the estimate for $\operatorname{Vol}_g(\partial M)$ follows by the same argument as in Proposition~\ref{vol_area}.
\end{proof}

We conclude this section with an observation on how the fact that a static potential is an eigenfunction of the Jacobi operator restricts the geometry of a static manifold with boundary. Recall that $\mathcal P$ denotes the vector space of static potentials on $(M,g)$.

\begin{theorem}
Let $(M^n,g)$, $n\geqslant3$, be connected and let $\partial_iM$ be a compact connected boundary component. Suppose that $\dim\mathcal P\geqslant 1$. If there exists a static potential which does not vanish on $\partial_iM$, then $\dim\mathcal P=1$.
\end{theorem}

This theorem strengthens Theorem 1.5 from~\cite{medvedev2026dimension}. Moreover, its proof is considerably simpler than that of Theorem 1.5.

\begin{proof}
Let $V$ be a static potential which does not vanish on $\partial_iM$. As we observe in the proof of Lemma~\ref{prop:positive-scalar-compactness}, every static potential $W$ is a first eigenfunction of the Jacobi operator for $\partial_iM$. Since the first eigenspace is one-dimensional, for every other static potential $W$ there exists a constant $c$ such that $W=cV$ on $\partial_iM$.

Set $U=W-cV$. Then $U=0$ on $\partial_iM$. Since both $W$ and $V$ satisfy the Robin boundary condition~\eqref{tracedb},
\[
\partial_\nu U=\frac{H_i}{n-1}U=0
\]
on $\partial_iM$.

Moreover, the static equation~\eqref{static} gives
\[
\operatorname{Hess}_gU
=
U\left(\operatorname{Ric}_g-\frac{R_g}{n-1}g\right).
\]
Let $\eta$ be an inward unit-speed geodesic normal to $\partial_iM$. Then
\[
\frac{d^2}{dt^2}U(\eta(t))
=
\left(\operatorname{Ric}_g(\dot\eta,\dot\eta)-\frac{R_g}{n-1}\right)U(\eta(t)).
\]
The initial conditions are
\[
U(\eta(0))=0,
\qquad
\frac{d}{dt}U(\eta(0))=0.
\]
Hence uniqueness for the corresponding ordinary differential equation implies that $U$ vanishes in a collar neighbourhood of $\partial_iM$.

Finally,
\[
\Delta_gU=-\frac{R_g}{n-1}U.
\]
By unique continuation, $U\equiv0$ on $M$. Therefore
\[
W=cV
\]
on $M$, proving the assertion.
\end{proof}

\begin{corollary} Let $(M^n,g)$, $n\geqslant3$, be a connected static manifold with boundary. If $\dim\mathcal P\geqslant2,$ then for every nonzero $V\in\mathcal P$ and every boundary component $\partial_iM$, $V^{-1}(0)\cap\partial_iM\neq\varnothing$.  
\end{corollary}

\section{Appendix}\label{appendix}

\subsection{Some identities}\label{identities}

In this section we derive some new identities for static manifolds with boundary admitting a nowhere-vanishing static potential. These formulas are quite simple to derive but we did not find them in the existing literature. We hope that they will be useful in the subsequent study of static manifolds with boundary admitting nowhere-vanishing potential. These identities can be used  for static manifolds without boundary if we omit boundary terms. 

\begin{proposition}\label{thm:bochner-type-identity}
Let $(M^n,g,V)$, $n\geqslant3$, be a compact connected static manifold with nonempty boundary, and suppose that the static potential does not vanish on $M$. Then
\[
    2\int_M\frac{|\nabla^g V|_g^2}{V^3}\,d\mu_g+\frac{R_g}{n-1}\int_M\frac{1}{V}\,d\mu_g=-\frac{1}{n-1}\sum_{i=1}^kH_i\int_{\partial_iM}\frac{1}{V}\,d\sigma_g.
\]
\end{proposition}

\begin{proof}
Since $M$ is connected and $V$ does not vanish on $M$, the function $V^{-1}$ is smooth on $M$. Using~\eqref{traced}, we compute
\[
\Delta_g\left(\frac{1}{V}\right)=-\frac{\Delta_gV}{V^2}+2\frac{|\nabla^gV|_g^2}{V^3}=\frac{R_g}{n-1}\frac{1}{V}+2\frac{|\nabla^gV|_g^2}{V^3}.
\]
On each boundary component $\partial_iM$, the static boundary condition~\eqref{tracedb} gives
\[
\partial_\nu\left(\frac{1}{V}\right)=-\frac{\partial_\nu V}{V^2}=-\frac{H_i}{n-1}\frac{1}{V}.
\]
Thus $V^{-1}$ satisfies
\[
\begin{cases}
\displaystyle
\Delta_g\left(\frac{1}{V}\right)=2\frac{|\nabla^gV|_g^2}{V^3}+\frac{R_g}{n-1}\frac{1}{V}& \text{in }M,\\[1em]
\displaystyle
\partial_\nu\left(\frac{1}{V}\right)=-\frac{H_i}{n-1}\frac{1}{V}& \text{on }\partial_iM.
\end{cases}
\]
Integrating the interior equation over $M$ and applying the divergence theorem yields
\[
2\int_M\frac{|\nabla^gV|_g^2}{V^3}\,d\mu_g+\frac{R_g}{n-1}\int_M\frac{1}{V}\,d\mu_g=\sum_{i=1}^k\int_{\partial_iM}\partial_\nu\left(\frac{1}{V}\right)d\sigma_g.
\]
Using the boundary condition for $V^{-1}$, we obtain
\[
2\int_M\frac{|\nabla^gV|_g^2}{V^3}\,d\mu_g+\frac{R_g}{n-1}\int_M\frac{1}{V}\,d\mu_g=-\frac{1}{n-1}\sum_{i=1}^kH_i\int_{\partial_iM}\frac{1}{V}\,d\sigma_g.
\]
\end{proof}

We establish another identity in the more general setting where $V$ may vanish in the interior of $M$. Although this identity can be proved in several ways, we follow the approach used to prove formula~(1.6) in~\cite{medvedev2024static}.

\begin{proposition}
Let $(M^n,g,V)$, $n\geqslant3$, be a compact connected static manifold with boundary, and suppose that $V$ does not vanish on $\partial M$. Let $\Omega$ be the closure of a connected component of $M\setminus V^{-1}(0)$, and write
\[
\partial\Omega=\left(\bigsqcup_{i=1}^{b}S_i\right)\sqcup\left(\bigsqcup_{\alpha=1}^{r}\Sigma_\alpha\right),
\]
where the $S_i$ are boundary components of $M$ and the $\Sigma_\alpha$ are components of $V^{-1}(0)$. Let $H_i$ denote the mean curvature of $S_i$ with respect to the outward unit normal, and let $\kappa_\alpha=|\nabla^gV|_g|_{\Sigma_\alpha}$, which is a positive constant. Then
\[
\begin{split} 
\int_\Omega&\left[V^3\left|\mathring{\operatorname{Ric}}_g\right|_g^2+\frac{3R_g}{n-1}V|\nabla^gV|_g^2-\frac{R_g^2}{n(n-1)}V^3\right]\,d\mu_g\\ &=\frac{1}{2(n-1)}\sum_{i=1}^{b}H_i\int_{S_i}\left(\operatorname{Ric}_g(\nu,\nu)-\frac{n+1}{(n-1)^2}H_i^2\right)V^3\,d\sigma_g\\ &\quad+\operatorname{sign}(V)\sum_{\alpha=1}^{r}\kappa_\alpha^3\operatorname{Vol}_g(\Sigma_\alpha), 
\end{split}
\]
where $\operatorname{sign}(V)\in\{-1,1\}$ denotes the constant sign of $V$ in the interior of $\Omega$.
\end{proposition}

\begin{proof}
After replacing $V$ by $-V$ if necessary, assume that $V>0$ in the interior of $\Omega$. 

Set
\[
c=\frac{R_g}{n-1},\qquad q=|\nabla^gV|_g^2,
\]
and consider the smooth vector field
\[
X=V^2\bigl(\operatorname{Ric}_g(\nabla^gV,\cdot)\bigr)^\sharp-q\nabla^gV.
\]
Using the contracted Bianchi identity, we compute
\[
\begin{split} \operatorname{div}_gX={}&2V\operatorname{Ric}_g(\nabla^gV,\nabla^gV)+V^2\langle\operatorname{Ric}_g,\operatorname{Hess}_gV\rangle_g\\ &-2\operatorname{Hess}_gV(\nabla^gV,\nabla^gV)-q\Delta_gV\\ ={}&V^3\left(|\operatorname{Ric}_g|_g^2-cR_g\right)+3cVq\\ ={}&V^3\left|\mathring{\operatorname{Ric}}_g\right|_g^2+\frac{3R_g}{n-1}V|\nabla^gV|_g^2-\frac{R_g^2}{n(n-1)}V^3. \end{split}
\]

Recall that by Proposition 1 in~\cite{cruz2023static}, $\kappa_\alpha=|\nabla^gV|_g|_{\Sigma_\alpha}$ is a positive constant and $\operatorname{Ric}_g(\nu,T)=0$ for every vector field $T$ tangent to $S_i$. Setting $a_i=H_i/(n-1)$, we obtain
\[
\langle X,\nu\rangle_g=a_i\operatorname{Ric}_g(\nu,\nu)V^3-a_iV|\nabla^{S_i}V|_g^2-a_i^3V^3.
\]
The tangential trace of the static equation gives
\[
\Delta_{S_i}V+(n-1)a_i^2V=-V\operatorname{Ric}_g(\nu,\nu).
\]
Multiplying by $V^2$ and integrating by parts on $S_i$, we find
\[
\int_{S_i}\operatorname{Ric}_g(\nu,\nu)V^3\,d\sigma_g=2\int_{S_i}V|\nabla^{S_i}V|_g^2\,d\sigma_g-(n-1)a_i^2\int_{S_i}V^3\,d\sigma_g.
\]
Consequently,
\[
\int_{S_i}\langle X,\nu\rangle_g\,d\sigma_g=\frac{H_i}{2(n-1)}\int_{S_i}\left(\operatorname{Ric}_g(\nu,\nu)-\frac{n+1}{(n-1)^2}H_i^2\right)V^3\,d\sigma_g.
\]

We next compute the flux through $\Sigma_\alpha$. Since $V>0$ on the interior side of $\Omega$ and $V=0$ on $\Sigma_\alpha$, the outward unit normal satisfies
\[
\nu=-\frac{\nabla^gV}{\kappa_\alpha},\qquad \partial_\nu V=-\kappa_\alpha.
\]
Although the first term in $X$ vanishes on $\Sigma_\alpha$, the second term does not:
\[
X=-|\nabla^gV|_g^2\nabla^gV\quad\text{on }\Sigma_\alpha.
\]
Therefore
\[
\langle X,\nu\rangle_g=-|\nabla^gV|_g^2\partial_\nu V=\kappa_\alpha^3,
\]
and hence
\[
\int_{\Sigma_\alpha}\langle X,\nu\rangle_g\,d\sigma_g=\kappa_\alpha^3\operatorname{Vol}_g(\Sigma_\alpha).
\]
Applying the divergence theorem on $\Omega$ and summing these boundary contributions proves the identity.
\end{proof}

\subsection{Static manifolds without boundary}

In this section we prove several results about complete static manifolds without boundary utilizing the technique of associated Einstein manifolds.

The following theorem is proved in \cite[Section 4.2]{reiris2015static} by using a different tool.

\begin{theorem}\label{thm:positive-scalar-curvature}
Let $(M^n,g,V)$, $n\geqslant 3$, be a connected complete static manifold without boundary. Suppose that the static potential $V$ does not vanish on $M$. Then $R_g\leqslant 0$.
\end{theorem}

\begin{proof}
Suppose, by contradiction, that $R_g>0$. Consider the associated Einstein manifold $(N=M\times S^1,h)$. It is clear that $(N,h)$ is complete and $\operatorname{Ric}_h=\dfrac{R_g}{n-1}h>0$. The Bonnet--Myers theorem therefore implies that $(N,h)$ is compact.

The canonical projection
\[
    \pi\colon N=M\times S^1\longrightarrow M
\]
is continuous and surjective. Hence $M=\pi(N)$ is compact. Since $(M,g)$ is complete and has no boundary, it is closed.

By~\eqref{traced}, $V$ is an eigenfunction of the Laplacian with eigenvalue $-\dfrac{R_g}{n-1}<0$. Since $V$ does not change sign, by the characterization of the first eigenfunction of the Laplacian on a closed connected manifold, $V$ must therefore be a first eigenfunction. Consequently, $V$ is constant and $\Delta_gV=0$. We arrive at a contradiction. Therefore $R_g\leqslant 0$.
\end{proof}

A related classical rigidity result is due to Lichnerowicz~\cite[p.~137]{Lichnerowicz1955}: a complete three-dimensional static manifold without boundary, admitting a static potential $V$ satisfying $V\to1$ at infinity, must be flat with $V\equiv1$. Anderson~\cite[Theorem~3.2]{Anderson1999} later strengthened this result by removing the asymptotic assumption on $V$.

We next consider the case $R_g<0$.

\begin{theorem}
Let $(M^n,g,V)$, $n\geqslant 3$, be a connected complete static manifold without boundary. Suppose that $R_g<0$. Then $(M,g)$ is not compact. Moreover, if $V>0$ on $M$, then it is unbounded above. 
\end{theorem}

\begin{proof}
Suppose, by contradiction, that $(M,g)$ is compact. Since $(M,g)$ is complete and has no boundary, it is closed. But then $V$ is a Laplace eigenfunction with \textit{positive} eigenvalue $-\dfrac{R_g}{n-1}$, which is impossible.

Now suppose that $V>0$. Set $\lambda:=-\dfrac{R_g}{n-1}>0$. Consider the associated Einstein manifold $(N,h)$. For the function $u:=\log V$, regarded as a function on $N$, one has
\[
\Delta_hu=\Delta_gu+\left\langle\nabla^g\log V,\nabla^gu\right\rangle_g=\frac{\Delta_gV}{V}=\lambda>0.
\]
If $u$ were bounded above, the Omori--Yau maximum principle~\cite{omori1967isometric,yau1975harmonic} on the complete manifold $(N,h)$ would provide a sequence of points $p_j\in N$ such that
\[
u(p_j)\longrightarrow\sup_Nu,\qquad \Delta_hu(p_j)\leqslant\frac1j.
\]
This contradicts $\Delta_hu\equiv\lambda>0$. Hence $u=\log V$ is unbounded above, and therefore $\sup_MV=+\infty$.
\end{proof}

\begin{example}
Let \((N^{n-1},h)\) be a complete Einstein manifold satisfying $\operatorname{Ric}_h=-a^2h$ for some constant \(a>0\). Consider the product manifold $M^n=\mathbb{R}\times N$ equipped with the metric $g=dt^2+h$ and define $V(t,x)=e^{at}$. Since both factors are complete, \((M,g)\) is complete. The Ricci tensor of the product metric is
\[
\operatorname{Ric}_g=0\cdot dt^2+\operatorname{Ric}_h=-a^2h.
\]
In particular, its scalar curvature is $R_g=-(n-1)a^2$.

Since \(V\) depends only on the \(\mathbb{R}\)-variable,
\[
\operatorname{Hess}_gV=a^2V\,dt^2 \qquad \text{and} \qquad \Delta_gV=a^2V.
\]
It follows that
\begin{align*}
(\Delta_gV)g+V\operatorname{Ric}_g=a^2V\bigl(dt^2+h\bigr)-a^2Vh=a^2V\,dt^2=\operatorname{Hess}_gV.
\end{align*}
Therefore, \((M,g,V)\) is a complete static manifold with everywhere positive static potential $V=e^{at}>0$.

In general, \(g\) does not have constant sectional curvature. For example, every mixed two-plane spanned by \(\partial_t\) and a vector tangent to \(N\) has sectional curvature zero. Thus, when \(a>0\), \((M,g)\) cannot be hyperbolic space, whose sectional curvature is strictly negative on every two-plane.

Moreover, by choosing \(N\) to be an Einstein manifold that does not have constant sectional curvature, one obtains further complete static manifolds with nowhere-vanishing static potential that are neither Euclidean nor hyperbolic.
\end{example}

\begin{proposition}\label{prop:conformally-flat-associated-metric}
Let $(M^n,g,V)$, $n\geqslant3$, be a static manifold, and let $(N^{n+1},h)$ be its associated Einstein manifold. If $h$ is locally conformally flat on its regular part, then $g$ has constant sectional curvature. More precisely, $\sec_g=\dfrac{R_g}{n(n-1)}$.
\end{proposition}

\begin{proof}
By the construction of the associated Einstein manifold, on its regular part one has $h=g+V^2d\theta^2$ and $\operatorname{Ric}_h=\dfrac{R_g}{n-1}h$.

Since $\dim N=n+1\geqslant4$ and $h$ is locally conformally flat, its Weyl tensor $W_h$ vanishes. The Riemann tensor decomposition implies, that an Einstein metric with vanishing Weyl tensor has constant sectional
curvature. 

It remains to relate the sectional curvatures of $h$ and $g$. Let $X,Y\in T_xM$ be $g$-orthonormal vectors. Regarding $X$ and $Y$ as horizontal vectors tangent to $N$, they are also $h$-orthonormal. Since $h=g+V^2d\theta^2$ is a warped product with one-dimensional fiber, its curvature tensor on horizontal vectors agrees with that of the base:
\[
    \operatorname{Rm}_h(X,Y,Y,X)=\operatorname{Rm}_g(X,Y,Y,X).
\]
Therefore
\[
    \sec_g(X\wedge Y)=\sec_h(X\wedge Y)=\frac{R_g}{n(n-1)}.
\]
Since $x$ and the two-plane $X\wedge Y\subset T_xM$ were arbitrary, $g$ has constant sectional curvature $\sec_g=\dfrac{R_g}{n(n-1)}$.

If the associated Einstein manifold is singular along the fixed-point set $\{V=0\}$, the preceding argument applies on the regular part $\{V\neq0\}$. The conclusion then extends to the whole of $M$ by continuity of the curvature tensor.
\end{proof}

\bibliography{mybib}
\bibliographystyle{alpha}

\end{document}